\documentclass[10pt, a4paper,
oneside, headinclude,footinclude]{scrartcl}
\usepackage[utf8]{inputenc}

\usepackage[
nochapters, % Turn off chapters since this is an article        
pdfspacing, % Makes use of pdftex’ letter spacing capabilities via the microtype package
dottedtoc % Dotted lines leading to the page numbers in the table of contents
]{classicthesis} % The layout is based on the Classic Thesis style
\usepackage{amsopn}
\usepackage[T1]{fontenc} % Use 8-bit encoding that has 256 glyphs
\usepackage{multirow}
\usepackage[utf8]{inputenc} % Required for including letters with accents
\usepackage{hhline}
\usepackage{graphicx} % Required for including images
\graphicspath{{Figures/}} % Set the default folder for images
\usepackage{indentfirst}
\usepackage{enumitem} % Required for manipulating the whitespace between and within lists
\usepackage{savesym}
\usepackage{mathrsfs}
\usepackage{geometry}
\usepackage{esint}
\usepackage{longtable}

\usepackage[
    backend=biber,
    style=numeric,
    natbib=false,
    url=false, 
    doi=false,
    eprint=false,
    maxnames=50
]{biblatex}

\usepackage{subfig} % Required for creating figures with multiple parts (subfigures)

\usepackage{amsmath, amssymb,amsthm,amsfonts} % For including math equations, theorems, symbols, etc

\usepackage[english,capitalize]{cleveref}

\usepackage{thmtools}
\usepackage{thm-restate}

\usepackage{hyperref}
\newcommand{\vertiii}[1]{{\left\vert\kern-0.25ex\left\vert\kern-0.25ex\left\vert #1 
    \right\vert\kern-0.25ex\right\vert\kern-0.25ex\right\vert}}

\theoremstyle{plain}
\newtheorem{teorema}{Theorem}[section]
\newtheorem{proposizione}[teorema]{Proposition}
\newtheorem{lemma}[teorema]{Lemma}

\newtheorem*{theorem*}{Theorem}

\theoremstyle{definition}
\newtheorem{definizione}{Definition}[section]

\theoremstyle{remark}
\newtheorem{osservazione}{Remark}[section]

\newcommand{\cc}{\text{Carnot-Carath\'eodory }}

\newcommand{\N}{\mathbb{N}}

\newcommand{\R}{\mathbb{R}}

\newcommand{\res}

\newcounter{const}

\newcounter{eps}

\DeclareMathOperator*{\dist}{dist}
\hypersetup{
colorlinks=true, breaklinks=true,bookmarksnumbered,
urlcolor=webbrown, linkcolor=RoyalBlue, citecolor=webgreen, % Link colors
pdftitle={}, % PDF title
pdfauthor={\textcopyright}, % PDF Author
pdfsubject={}, % PDF Subject
pdfkeywords={}, % PDF Keywords
pdfcreator={pdfLaTeX}, % PDF Creator
pdfproducer={LaTeX with hyperref and ClassicThesis} % PDF producer
} 

\title{\normalfont\spacedallcaps{
Unextendable intrinsic Lipschitz curves}} 
\author{\spacedlowsmallcaps{Gioacchino Antonelli\textsuperscript{*} and Andrea Merlo\textsuperscript{**}}}

\date{}

\begin{document}

\renewcommand{\sectionmark}[1]{\markright{\spacedlowsmallcaps{#1}}} 
\lehead{\mbox{\llap{\small\thepage\kern1em\color{halfgray} \vline}\color{halfgray}\hspace{0.5em}\rightmark\hfil}} 
\pagestyle{scrheadings}
\maketitle 
\setcounter{tocdepth}{2}
\paragraph*{Abstract}
In the setting of Carnot groups, we exhibit examples of intrinisc Lipschitz curves of positive $\mathcal{H}^1$-measure that intersect every connected intrinsic Lipschitz curve in a $\mathcal{H}^1$-negligible set. As a consequence such curves cannot be extended to connected intrinsic Lipschitz curves. 

The examples are constructed in the Engel group and in the free Carnot group of step 3 and rank 2. While the failure of the  Lipschitz extension property was already known for some pairs of Carnot groups, ours is the first example of the analogous phenomenon for intrinsic Lipschitz graphs. This is in sharp contrast with the Euclidean case.

{\let\thefootnote\relax\footnotetext{* \textit{Scuola Normale Superiore, Piazza dei Cavalieri, 7, 56126 Pisa, Italy,}}}
{\let\thefootnote\relax\footnotetext{** \textit{Université Paris-Saclay, 307 Rue Michel Magat Bâtiment, 91400 Orsay, France.}}}

\paragraph*{Keywords} 
Lipschitz extension property, Carnot groups, Engel groups, free Carnot groups, intrinisc Lipschitz graphs

\paragraph*{MSC (2010)} 53C17, 22E25, 28A75, 49Q15, 26A16.

%\tableofcontents

\section{Introduction}
Extending Lipschitz maps is a fundamental tool in Geometric Measure Theory. We say that the pair of metric spaces $(X,Y)$ has the {\em Lipschitz extension property}, LEP from now on, if for any Lipschitz map $f:A\subset X\to Y$ defined on a subset $A$ of $X$, there exists another Lipschitz map $\widetilde f:X\to Y$ such that $\widetilde f|_A=f$. It is worth recalling that the LEP holds true in case $X$ and $Y$ are Hilbert spaces and in the case in which at least one between $X$ and $Y$ is a finite dimensional Banach space: if the target $Y$ is finite dimensional, then the LEP follows from the standard Mc Shane's extension theorem. If on the other hand is the domain $X$ to be finite dimensional, then the LEP for the couple $(X,Y)$ follows from \cite[Theorem 2]{MR852474}.

The study of Geometric Measure Theory on Carnot groups, initiated in the seminal works \cite{AK00, Serapioni2001RectifiabilityGroup}, has brought the attention to the investigation of LEP for pairs of Carnot groups. Carnot groups are a natural generalization of Euclidean spaces. Indeed, (quotients of) Carnot groups arise both as the infinitesimal models of subRiemannian manifolds and as asmyptotic models of Riemannian Lie groups. For basic definitions we refer the interested reader to \cite{LD17}. We stress that there is not a general way to understand whether an arbitrary couple of Carnot groups has the LEP. However, some specific results are known, see \cite{BaloghFassler09, Magnani10, WengerYoung10}.

Despite the many analogies with Euclidean spaces, Carnot groups also have many insidious features. For instance, in sharp contrast with the Euclidean spaces, they are purely unrectifiable with respect to Lipschitz surfaces of dimension bigger than the dimension of the first layer of their Lie algebra, see \cite{MagnaniUnrect}. Therefore, in order to find a good substitute of the class of Lipschitz graphs in the realm of Carnot groups, Franchi, Serapioni and Serra Cassano introduced the notion of intrinsically Lipschitz graph, see \cite{FSSC06}. 
% This notion was introduced also with the aim of finding a well-behaved notion of rectifiability in Carnot groups.
The idea of the construction of these surfaces is to say that a graph of a function between two homogeneous complementary subgroups of a Carnot group is intrinsically Lipschitz whenever it satisfies a uniform cone-like property at every point, where the cones are intrinsic, cf. \cref{def:IntrinsicLipschitzGraph}. We refer to
\cite{FranchiSerapioni16} for a wide study of the notion of intrinsically Lipschitz functions and graphs. 

We stress that the study of intrinsically Lipschitz graphs in Carnot groups is a very active area of research today, with several important contributions. See, e.g., \cite{ChousionisFasslerOrponen19} for relations of this notion with quantitative rectifiability, \cite{Vittone20} for a proof of Rademacher theorem for co-horizontal intriniscally Lipschitz graphs in the Heisenberg groups, \cite{Foliated20} for a deep study of structural properties of intriniscally Lipschitz graphs in the first Heisenberg group $\mathbb H^1$, and \cite{DLDMV19} for rectifiability results with intrinsically Lipschitz graphs.

Questions about LEP can be proposed also for intriniscally Lipschitz graphs. For example, is it true that every intrinsically Lipschitz map $\varphi:U\subseteq \mathbb W\to\mathbb V$, where $\mathbb W,\mathbb V$ are complementary subgroups of a Carnot group, can be extended to an \textbf{entire} intrinsically Lipschitz map $\widetilde\varphi:\mathbb W\to\mathbb V$? The answer is positive when the subgroup $\mathbb V$ is \textbf{horizontal}, i.e., contained in the first layer of the stratification of the Carnot group, cf. \cite[Theorem 1.5]{Vittone20}, and \cite[Theorem 4.25]{FSSC11}.

Hence, we have the validity of the LEP for low-codimensional intriniscally Lipschitz graphs in arbitrary Carnot groups. Moreover, in the recent \cite{KatrinDaniela}, the authors prove that every $\varphi:U\subseteq \mathbb W\to\mathbb V$, where $\mathbb W,\mathbb V$ are complementary subgroups of the $n$-th Heisenberg group $\mathbb H^n$, and $\mathbb W$ is horizontal, can be extended to an entire intrinsically Lipschitz map $\widetilde\varphi:\mathbb W\to\mathbb V$, cf. \cite[Theorem 1.2]{KatrinDaniela}.

In this note we show that the previous example is special. Namely, we provide a negative answer to the validity of the LEP for intrinsically Lipschitz maps defined on subsets of horizontal subgroups of a Carnot group. This is the first example in which the LEP of intrinisc Lipschitz graphs is known to fail on Carnot groups. 

We recall that with $\mathbb F_{2,3}$ we denote the free Carnot group of rank 2 and step 3, and with $V_1$ we denote its horizontal layer, see \cref{sec:Preliminaries}. Up to a choice of an adapted basis $\mathscr{B}:=\{X_1,X_2,X_3,X_4,X_5\}$ of the Lie algebra as in \cref{sec:Preliminaries}, we identify $\mathbb F_{2,3}$ with $\mathbb R^5$ through the exponential map. We endow the Lie algebra of $\mathbb F_{2,3}$ with an auxiliary inner product that makes $\mathscr{B}$ an orthonormal basis, and we fix an arbitrary left-invariant homogeneous distance on $\mathbb F_{2,3}$, see \cref{sec:Preliminaries} for more details. The Hausdorff measures on $\mathbb F_{2,3}$ are computed with respect to such a distance. Finally, for every $e\in V_1$ we denote $\mathfrak{N}(e):=\{\exp(te):t\in\mathbb R\}$, and $\mathbb V(e):=\exp(e^\perp)$. 

We are now ready to state our result. For the proof of the forthcoming statement, see \cref{prop.intersez.curves}, and \cref{thm:EveryDirection}.

\begin{teorema}\label{thm:EveryDirectionINTRO}
Let $\mathbb F_{2,3}$ be the free Carnot group of rank 2 and step 3, and let $V_1$ be the first layer of a stratification of its Lie algebra. For any $e\in V_1$ there exists a compact set $K\subseteq \mathfrak{N}(e)$, and an intriniscally Lipschitz function $\varphi:K\to \mathbb V(e)$ such that the following two conditions hold.
\begin{itemize}
    \item[(i)] $\mathcal{H}^1(\mathrm{graph}(\varphi))>0$, where $\mathrm{graph}(\varphi):=\{a\cdot\varphi(a):a\in K\}$,
    \item[(ii)] for any intrinsically Lipschitz map $\widetilde\varphi:\Omega\to \mathbb{V}(e)$, where $\Omega$ is an open subset of $\mathfrak{N}(e)$, we have
    $$
    \mathcal{H}^1(\mathrm{graph}(\varphi)\cap \mathrm{graph}(\widetilde\varphi))=0.
    $$
\end{itemize}
As a consequence, there exists no intrinsically Lipschitz map $\psi:\mathfrak{N}(e)\to\mathbb V(e)$ such that $\psi|_K=\varphi$.
\end{teorema}

Notice that \cref{thm:EveryDirectionINTRO} tells that the examples for which LEP fails can be constructed for every horizontal direction in $\mathbb F_{2,3}$. We stress that the previous example can also be constructed in the Engel group taking as $e$ the unique horizontal abnormal direction, see \cref{rem:Engel}. 

We highlight a connection between the LEP for low-dimensional intrinsically Lipschitz graphs in Carnot groups and the LEP for pairs of Carnot groups, which has already been noticed in \cite{KatrinDaniela}. It is readily seen that whenever $\varphi:U\subseteq \mathbb W\to\mathbb V$ is an intriniscally Lipschitz function, where $\mathbb W,\mathbb V$ are complementary subgropus of a Carnot group $\mathbb G$, and $\mathbb V$ is normal, hence the graph map $\mathrm{graph}(\varphi):U\to\mathbb G$ is Lipschitz, see \cite[Proposition 3.7]{FranchiSerapioni16}, and $\mathbb W$ is a Carnot subgroup, see \cite[Remark 2.1]{AM20}. Hence asking for the validity of LEP for intrinsically Lipschitz maps between complementary subgroups $\mathbb W$ and $\mathbb V$, with $\mathbb V$ normal, amounts at asking if the couple of Carnot groups $(\mathbb W,\mathbb G)$ has a LEP in such a way that the extension preserves the graph structure associated to $(\mathbb W,\mathbb V)$.

We briefly discuss the proof of \cref{thm:EveryDirectionINTRO}. The first step is to prove that if a Lipschitz curve starting from $0$ and defined on a compact interval of $\mathbb R$ with values in a Carnot group has intrinsic derivative that stays at every time in some cone, then the entire curve lies in the closure of the semigroup generated by that cone, see \cref{l:curves:Snu}. Hence, we exploit the explicit expression of the semigroup computed in \cite[Proposition 5.7]{bellettini2019sets} for $\mathbb F_{2,3}$ in order to give a geometric constraint on every Lipschitz curve in $\mathbb F_{2,3}$ that has intrinsic derivative that stays at every time in some precise one-sided Euclidean cone, see \cref{propconoXcurv}.

Hence, we carefully construct a biLipschitz curve $\gamma$ from a compact $\mathcal{H}^1$-positive measured set $K\subseteq \mathfrak{N}(X_2)$ to $\mathbb F_{2,3}$
that in coordinates reads as $\gamma(t)=(0,t,0,\gamma_4(t),0)$ with a strictly decreasing $\gamma_4$, see \cref{curvunr}. The fact that $\gamma_4$ is strictly decreasing, together with the explicit expression of the semigroup computed in \cite{bellettini2019sets}, implies that every Lipschitz curve from a compact interval of $\mathfrak{N}(X_2)$ to $\mathbb F_{2,3}$ that has intrinsic derivative that stays at every time in some precise one-sided Euclidean cone could intersect $\gamma$ at most in one point, see \cref{unrunr}. Hence, to obtain \cref{thm:EveryDirectionINTRO} with $e=X_2$, it is sufficient to notice that the construction in \cref{curvunr} ensures that the projection $\mathfrak{N}(X_2)\supset K\ni t \mapsto P_{\mathbb V(X_2)}(\gamma(t))$ is intrinsically Lipschitz, and further notice that the graph of every intrinsically Lipschitz map from a compact interval of $\mathfrak{N}(X_2)$ to $\mathbb V(X_2)$ is a Lipschitz curve that has intrinsic derivative that stays at every time in some one-sided Euclidean cone, see the discussion after \cref{def:IntrinsicLipschitzGraph}.

Finally, to obtain the example of \cref{thm:EveryDirectionINTRO} in every direction $e\in V_1$, we exploit the fact that $\mathbb F_{2,3}$ is a free Nilpotent Lie group, cf. \cref{pautomorf}.

\medskip

Before concluding the introduction let us briefly discuss why this counterexample is so relevant. In \cite{kircharea} one of the core observations that yields both the existence of the density and the area formulae is that given a Lipschitz function $f:B\to X$, where $X$ is an arbitrary metric space and $B$ is a Borel subset of $\R^n$, one can without loss of generality assume that the image under $f$ of the set where the metric differential is injective, is covered $\mathcal{H}^n$-almost all by the image of countably many Lipschitz maps of arbitrary small Lipschitz constant defined on open sets of suitable $n$-dimensional Banach spaces, see \cite[Lemma 4]{kircharea}. This approach is the natural extension to the metric setting of the classical one, see for instance \cite[Lemma 3.2.17]{Federer1996GeometricTheory}.
The counterexample we construct in this note therefore tells us that the area formula, the existence of the density and the general structure theory for intrinsic Lipschitz rectifiable sets cannot be deduced from a theory for the graphs defined on open sets. In particular the delicate techniques employed in \cite{AntonelliMerlo2021, antonelli2020rectifiable} to deal with intrinsic Lipschitz functions defined on Borel sets were an unavoidable evil in order to obtain those results at that level of generality.

\section{Preliminaries}\label{sec:Preliminaries}

For general facts about Carnot groups, we refer the reader to \cite{LD17}. In this note we work specifically in the Carnot group $\mathbb{F}_{2,3}$ that is the free group of step 3 and rank 2. The Lie algebra of $\mathbb{F}_{2,3}$ is 5 dimensional and it is generated by 2 vectors, that we denote by $X_1$ and $X_2$. We will denote $V_1:=\mathrm{span}\{X_1,X_2\}$, and we endow it with a scalar product denoted by $\langle\cdot,\cdot\rangle$ that makes $X_1,X_2$ orthonormal. A basis of the Lie algebra is completed to $X_1,\ldots,X_5$ for which the only non-trivial Lie brackets are
$$
[X_2,X_1]=X_3,\qquad[X_3,X_1]=X_4,\qquad [X_3,X_2]=X_5.
$$
Every graded group, and so $\mathbb F_{2,3}$, has a one-parameter family of dilations that we denote by $\{\delta_\lambda: \lambda >0\}$. 
We will indicate with $\delta_{\lambda}$ both the dilation of factor $\lambda$ on the group and its differential.

The operation of $\mathbb{F}_{2,3}$ in exponential coordinates \textbf{of the second type} (cf. \cite[Section 5]{bellettini2019sets}) is given by the following expression
$$
x\cdot y=\Big(x_1+y_1,x_2+y_2,x_3+y_3-x_1y_2,x_4+y_4-x_1y_3+\frac{x_1^2y_2}{2},x_5+y_5+x_1x_2y_2+\frac{x_1y_2^2}{2}-x_2y_3\Big),
$$
see \cite[Equation (5.1)]{bellettini2019sets} for a reference, and where as usual $x_i$ and $y_i$ are the coordinates of $x$ and $y$ seen as vectors in $\mathbb R^5$.

\begin{definizione}[Smooth-box metric]\label{smoothnorm}
Let $\exp$ be the exponential map of $\mathbb F_{2,3}$. We identify $\mathbb F_{2,3}$ with $\mathbb R^5$ through $\exp$ and the previous choice of the basis $\{X_1,\dots,X_5\}$ of the Lie algebra.

For any $g\in \mathbb{F}_{23}\equiv \mathbb R^5$, we let
$$\lVert (g_1^1,g_1^2,g_2,g_3^1,g_3^2)\rVert=\lVert \exp(g_1^1X_1+g_1^2X_2+g_2X_3+g_3^1X_4+g_3^2X_5)\rVert:=\max\{\varepsilon_1\lvert g_1\rvert,\varepsilon_2\lvert g_2\rvert^{1/2},\varepsilon_3|g_3|^{1/3}\},$$
where $g_1:=(g_1^1,g_1^2)$, $g_3:=(g_3^1,g_3^2)$, and $\varepsilon_1=1,\varepsilon_2,\varepsilon_3$ are suitably small parameters depending only on the group $\mathbb F_{23}$. For the proof of the fact that we can choose the $\varepsilon_i$'s in such a way that $\lVert\cdot\rVert$ is a left invariant, homogeneous norm on $\mathbb F_{23}$ that induces a left-invariant homogeneous distance we refer to \cite[Section 5]{step2}. Furthermore, we define
$$
d(x,y):=\lVert x^{-1}\cdot y\rVert.
$$
%and let $U(x,r):=\{z\in \mathbb{G}:d(x,z)<r\}$ be the open metric ball relative to the distance $d$ centred at $x$ at radius $r>0$. The closed ball will be denoted with $B(x,r):=\{z\in \mathbb{G}:d(x,z)\leq r\}$. Moreover, for a subset $E\subseteq \mathbb G$ and $r>0$, we denote with $B(E,r):=\{z\in\mathbb G:\dist(z,E)\leq r\}$ the {\em closed $r$-tubular neighborhood of $E$}  and with $U(E,r):=\{z\in\mathbb G:\dist(z,E)< r\}$ the {\em open $r$-tubular neighborhood of $E$}.
Notice that 
\begin{equation}\label{eqn:NormaCoordinate}
\|\exp(x_2X_2)\exp(x_4X_4)\|=\|\exp(x_4X_4)\exp(x_2X_2)\|=\|\exp(x_2X_2+x_4X_4)\|=\max\{\varepsilon_1|x_2|,\varepsilon_3|x_4|^{1/3}\},
\end{equation}
where we are using that $[X_2,X_4]=0$. Hence, the vector $(0,x_2,0,x_4,0)$ defines the same element of $\mathbb F_{23}$ no matter if we read it in coordinates induced by the identification through $\exp$ and the basis $\{X_1,\dots,X_5\}$, or if we read it in exponential coordinates \textbf{of the second type}. Hence, when computing his norm we will freely use \eqref{eqn:NormaCoordinate}.
\end{definizione}

We recall without proof the following lemma which is essentially due to the fact that $\mathbb F_{2,3}$ is a free nilpotent Lie group.
\begin{lemma}\label{pautomorf}
For any couple of linearly independent vectors $Y_1,Y_2$ in the first layer $V_1$ of the Lie algebra of $\mathbb{F}_{2,3}$ there exists an homogeneous group automorphism $\Psi:\mathbb{F}_{2,3}\to\mathbb{F}_{2,3}$ such that
$$\Psi(\mathrm{exp}(X_i))=\mathrm{exp}(Y_i)\qquad\text{for any $i=1,2$}.$$
\end{lemma}

Let us now give some definitions.
\begin{definizione}[Vector fields in coordinates]\label{campii}
The vector fields $\{X_i\}_{i=1,\ldots,5}$ act on smooth functions, and for $i=1,\ldots,5$ we can write them in exponential coordinates of the second type as
$$
X_i(x):=\sum_{j=1}^5 \mathfrak{c}_j^i(x)\partial_j,
$$
where, for every $j=1,\dots,5$, $\partial_j=\partial_{x_j}$ are the standard derivations on $\mathbb R^5$ associated to the coordinate functions, and where $\mathfrak{c}_j^i(x)$ are polynomials.
We shall order the coefficients $\mathfrak{c}_j^i$ relative to the vector-fields $X_1,X_2$ in the form of the matrix
$$
\mathscr{C}(x):=\begin{pmatrix}
\mathfrak{c}_1^1(x) & \mathfrak{c}^{2}_1(x)\\
\vdots&\vdots\\
\mathfrak{c}^1_5(x)&\mathfrak{c}^{2}_5(x)
\end{pmatrix}.	
$$
\end{definizione}

\begin{definizione}[Derivative of a curve]\label{def:Derivative}
Let $B$ be a Borel subset %{\color{red} FORSE SERVE BOREL (OPPURE BOREL LIMITATO)} 
of the real line. Given a curve $\gamma:B\to \mathbb F_{2,3}\equiv \mathbb R^5$ and a Lebesgue density point $t\in B$, we denote
$$
\gamma^\prime(t):=\lim_{\substack{r\to 0 \\ t+r\in B}}\frac{\gamma(t+r)-\gamma(t)}{r},\qquad \text{whenever the right-hand side exists.}
$$

Furthermore, we say that an absolutely continuous curve $\gamma:[a,b]\to \mathbb F_{2,3}$ is \emph{horizontal} if there exists a measurable function $h:[a,b]\to V_1$ such that
\begin{itemize}
    \item[(i)] $\gamma^\prime(t)=\mathscr{C}(\gamma(t))[h(t)]$ for $\mathcal{H}^1$-almost every $t\in [a,b]$,
    \item[(ii)] $\lvert h\rvert\in L^\infty([a,b])$.
\end{itemize}
Following the notation of \cite{tesimonti} we shall refer to $h$ as the \emph{canonical coordinates of }$\gamma$ and if $\lVert h\rVert_\infty\leq 1$ we will say that $\gamma$ is a \emph{sub-unit} path.
Finally, we define the Carnot-Carath\'eodory distance $d_c$ on $\mathbb F_{2,3}$ as
\begin{equation}
d_c(x,y):=\inf\{T\geq 0:\text{ there is a sub-unit path }\gamma:[0,T]\to\R^5\text{ such that }\gamma(0)=x\text{ and }\gamma(1)=y\}.\nonumber
\end{equation}
It is well known that $d_c(\cdot,\cdot)$ is a left invariant homogeneous metric on $\mathbb F_{2,3}$. 
Moreover, $d_c$ on $\mathbb F_{2,3}$ is biLipschitz equivalent to the distance induced by the norm $\|\cdot\|$ introduced in \cref{smoothnorm}, since every two homogeneous left-invariant distances are biLipschitz equivalent on a Carnot group. %Finally throughout the paper we will often denote by $\lVert \cdot\rVert$ the homogeneous function $x\mapsto d_c(\cdot,0)$.
\end{definizione}

\begin{definizione}[Hausdorff Measures]\label{def:HausdorffMEasure}
Throughout the paper we define the $h$-dimensional {\em Hausdorff measure}\label{hausmeas} relative to $d$ as
$$
\mathcal{H}^h(A):=\sup_{\delta>0}\inf \left\{\sum_{j=1}^{\infty} 2^{-h}(\mathrm{diam}E_j)^h:A \subseteq \bigcup_{j=1}^{\infty} E_j,\, \mathrm{diam}E\leq \delta\right\}.
$$
%We define the $h$-dimensional {\em centered Hausdorff measure} relative to $d$ as\label{centredhausmeas}
%$$
%\mathcal{C}^{h}(A):=\underset{E\subseteq A}{\sup}\,\,\mathcal{C}_0(E),
%$$
%where
%$$
%\mathcal{C}^{h}_0(E):=\sup_{\delta>0}\inf\bigg\{\sum_{j=1}^\infty  r_j^m:E\subseteq \bigcup_{j=1}^\infty B(x_j,r_j),~ x_j\in E,~r_j\leq\delta\bigg\}.
%$$
%We remark that $\mathcal{C}^h$ is an outer measure, and thus it defines a Borel regular measure, see \cite[Proposition 4.1]{EdgarCentered}, and that the measures $\mathcal{H}^h,\mathcal{C}^h$ are all equivalent measures, see \cite[Section 2.10.2]{Federer1996GeometricTheory} and \cite[Proposition 4.2]{EdgarCentered}.
\end{definizione}

We now give two lemmas about Lipschitz curves that will be useful in the proof of the main result of this note.
The following lemma allows us to characterise those Euclidean Lipschitz curves that are also Lipschitz curves when $\R^5$ is endowed with the \cc distance $d_c$ introduced above. Its proof is based on an extension argument and on \cite[Lemma 1.3.3]{tesimonti}. Since the proof is standard, we omit it.

\begin{lemma}\label{lemma.monti1}
Let $B$ be a Borel subset of the real line. If a curve $\gamma:B\to\mathbb F_{2,3}$ is $L$-Lipschitz with respect to the distance $d_c$ on $\mathbb F_{2,3}$, then the Euclidean derivative $\gamma'(t)$, see \cref{def:Derivative}, exists at almost every $t\in B$, and is such that 
$$
\gamma^\prime(t)=\mathscr{C}(\gamma(t))[h(t)]\text{ for }\mathcal{H}^1\text{-almost every }t\in B,
$$
for some $h\in L^\infty(B,V_1)$ with $\lVert h\rVert_\infty\leq L$.
\end{lemma}

The following lemma can be proved with an extension argument and basing on \cite[Lemma 2.1.4]{tesimonti}. Again, since the proof is standard, we omit it.
\begin{lemma}\label{lemma.monti2}
Let $B$ be a Borel subset of the real line and assume $\gamma:B\to \mathbb F_{2,3}$ is a Lipschitz curve with respect to the metric $d_c$. If $h\in L^\infty(B,V_1)$ is the vector of canonical coordinates of $\gamma^\prime$, then for $\mathcal{H}^1$-almost every $t\in B$ we have
    $$
    D\gamma(t):=\lim_{\substack{s\to 0\\t+s\in B}}\delta_{1/s}(\gamma(t)^{-1}\cdot\gamma(t+s))=(h_1(t),h_2(t),0,0,0).
    $$
    In particular $D\gamma(t)$ exists for $\mathcal{H}^1$-almost every $t\in B$.
\end{lemma}

% {\color{Blue}
% \begin{proof}

% Let us first prove the result in the case $B$ is a compact set. Arguing as in the proof of Lemma \ref{lemma.monti1} we can extend the curve $\gamma$ to a curve $\tilde{\gamma}$ defined on the segment that is the convex hull of $B$ and thanks to \cite[Lemma 2.1.4]{tesimonti}, we infer that $D\tilde{\gamma}(t)=(h(t),0,\ldots,0)$ for $\mathcal{H}^1$-almost every $t$ in the convex hull of $B$. This implies in particular that:
% $$\lim_{\substack{s\to 0\\ t+s\in B}}\delta_{1/s}(\gamma(t)^{-1}\cdot\gamma(t+s))=\lim_{\substack{s\to 0\\ t+s\in B}}\delta_{1/s}(\tilde{\gamma}(t)^{-1}\cdot\tilde{\gamma}(t+s))=(\tilde{h}_1(t),\ldots,\tilde{h}_{n_1}(t),0,\ldots,0),$$
% for $\mathcal{H}^1$-almost any $t\in B$, where $\tilde{h}$ is the vector of canonical coordinates for $\tilde{\gamma}$. However, since $\gamma^\prime$ and $\tilde{\gamma}^\prime$ coincide almost everywhere on $B$ and $\mathscr{C}(x)$ is an injective linear map for any $x$, one also has that $\tilde{h}=h$ almost everywhere on $B$. This concludes the proof in the case $B$ is compact. 

% The same approximation argument we employed in Lemma \ref{lemma.monti1} in order to pass from compacts to Borel sets concludes the proof.
% \end{proof}
% }

Let us now give some basic definitions of cones.
\begin{definizione}[Cone $C(e,\sigma)$]
\label{C-cone}
Let $e\in \mathrm{span}\{X_1,X_2\}$ be a unit vector and $\sigma\in (0,1)$. We denote by $C(e,\sigma)$ the one-sided, open, convex cone with axis $e$ and opening $\sigma$, namely
$$
C(e,\sigma):=\{x\in V_1:\langle x,e\rangle >(1-\sigma^2)|x|\}.
$$
Let $B$ be a Borel subset of the real line, and let us now shorten the notation to $C:=C(e,\sigma)$. A Lipschitz curve $\gamma:B\to \mathbb F_{2,3}$, where as usual $\mathbb F_{2,3}$ is endowed with the metric $d_c$, is said to be a $C$-curve if 
  % \item[(\hypertarget{dir.cono}{i})] ,
  %\text{ $\pi_1(\gamma(s))-\pi_1(\gamma(t))\in C$ for any $t,s\in B$ with $t<s$.}
  \begin{equation}
      D\gamma(t)\in C\,\,\text{for $\mathcal{H}^1$-almost every}\, t\in B.
      \label{conecondition}
  \end{equation}
Notice that a Lipschitz cuve is always Pansu-differentiable almost everywhere, see \cite{Pansu}, hence the previous $D\gamma(t)$ exists for almost every $t\in B$. If the domain of a $C$-curve $\gamma$ is a compact interval, we will say that $\gamma$ is a \emph{full} $C$-curve.
\end{definizione}

\begin{definizione}[$\mathfrak{N}(e)$ and $\mathbb V(e)$]
For any $e\in V_1$, in the following we will always denote by
\begin{itemize}
    \item[(i)]$\mathfrak{N}(e)$ the $1$-parameter subgroup tangent to $e$ at $0$, i.e., $\mathfrak{N}(e):=\exp(\{te:t\in\mathbb R\})$,
    \item[(ii)] $\mathbb{V}(e)$ the hyperplane orthogonal to $e$ (in the Euclidean sense), i.e., $\mathbb V(e):=\exp(e^\perp)$. Notice that $\mathbb V(e)$ is also a normal homogeneous $9$-dimensional subgroup of $\mathbb F_{2,3}$.
\end{itemize}
\end{definizione}

\begin{definizione}[Cone $K(e,\sigma)$ and semigroup $X(e,\sigma)$]\label{def:conimet}
For any $e\in  \mathrm{span}\{X_1,X_2\}$ and $\sigma\in (0,1)$ we let
\begin{itemize}
    \item[(i)] $X(e,\sigma):=\{\prod_{i=1}^N\delta_{t_i}(v_i):v_i\in C(e,\sigma),\text{ }N\in\N\text{ and }t_i>0\}$,
    \item[(ii)]$ K(e,\sigma):=\{w\in\mathbb F_{2,3}:\dist(w,\mathfrak{N}(e))\leq \sigma \lVert w \rVert\}$.
\end{itemize}
\end{definizione}

Let us remark that the two above notions of cones rise from two different aspects of the nature of the group $\mathbb F_{2,3}$. On the one hand, $X$ is a cone that arises from the algebraic structure of $\mathbb F_{2,3}$, it characterises the points that can be \emph{reached} from the origin by a continuous, piece-wise linear path that goes in the direction of the Euclidean cone $C(e,\sigma)$. On the other hand,  $K(e,\sigma)$ is a metric cone and as such is more suitable for the local description of geometric ($1$-dimensional objects) inside the group $\mathbb F_{2,3}$.

Let us now state the following result which will be of crucial importance in the proof of the result of this note. The following result can be proved by using \cite[Lemma 2.1.4]{tesimonti} and an approximation result as in \cite[Lemma 3.2]{DLDMV19}. We write the proof for the reader's convenience.

\begin{lemma}
	\label{l:curves:Snu}
	Let $ T>0 $, $e\in V_1$, and suppose $\gamma: [0,T]\to\mathbb F_{2,3}$ is a Lipschitz curve such that $ \gamma(0)=0 $. If $ \langle D\gamma(t),e \rangle > (1-\sigma^2)\lvert  D\gamma(t)\rvert $ for almost every $ t\in [0,T]$, then $ \gamma(T) \in \mathrm{cl}(X(e,\sigma))$.
\end{lemma}

\begin{proof}
Thanks to \cite[Proposition 1.3.3]{tesimonti} there exists a function $h\in L^\infty([0,T],V_1)$ such that
$$
\gamma^\prime(t)=\mathscr{C}(\gamma(t))[h(t)]=\sum_{j=1}^{2}h^j(t)X_j(\gamma(t)).
$$
	for $\mathcal{H}^1$-almost every $t\in [0,T]$. Then, by \cite[Lemma 2.1.4]{tesimonti} we know that $\langle h(t),e\rangle>(1-\sigma^2)\lvert  h(t)\rvert$ for $\mathcal{H}^1$-almost every $t\in [0,T]$. Let $\{h_i\}_{i\in\N}$ be a sequence of piece-wise constant curves in $L^\infty([0,T];V_1)$ such that
	\begin{itemize}
	    \item[(i)]$\langle h_i(t),e\rangle\geq(1-\sigma^2)\lvert h_i(t)\rvert$ for $\mathcal{H}^1$-almost every $t\in[0,T]$ and any $i\in\N$,
	    \item[(ii)]$\lim_{i\to\infty}\lVert h_i-h\rVert_{L^1([0,T];\mathbb R^{2})}=0$.
	\end{itemize}
%(this can be done for example by a standard construction following Theorem 2.22 of \emph{dispensa Degiovanni 98} where actually the convergence is pointwise almost everywhere)
We can also assume that $\sup_{i\in \mathbb N}\sum_{j=1}^{2}\|h_i^j\|_\infty\leq M$, for some $M>0$. According to to the definition of  $X(e,\sigma)$ and since the $h_i$'s are piece-wise constant, the curves $\gamma_i$ solving the Cauchy problems
	\[
	\begin{cases}
	\dot\gamma_i(t)=\sum_{j=1}^{2} h_i^j(t)X_j(\gamma_i(t)) & \text{ for $\mathcal{H}^1$-almost every $t\in [0,T]$,}\\
	\gamma_i(0)=0,
	\end{cases}
	\]
	are such that $\gamma_i(t)\in \mathrm{cl}(X(e,\sigma))$ for any $t\in [0,T]$. In addition, since $d(\gamma_i(t),0)\leq Mt$ for every $t\in [0,T]$, there exists a compact set $K\subseteq \mathbb{G}$ for which 
	\[
	\bigcup_{i\in \mathbb N} \gamma_i([0,T])\cup \gamma([0,T])\subseteq K.
	\]
	We now claim that
	\begin{equation}
	    	\lim_{i\to \infty} d(\gamma_i(t),\gamma(t))=0,
	    	\label{eq:num:num10}
	\end{equation}
	for every $t\in [0,T]$. If we assume that \eqref{eq:num:num10} holds true, by continuity we inter that $\gamma(t)\in \mathrm{cl}(X(e,\sigma))$, and this concludes the proof.
	Let us fix $t\in [0,T]$ and compute
	\[
	\begin{aligned}
	|\gamma_i(t)-\gamma(t)|&=\left|\int_0^t\sum_{j=1}^{2} \left(h_i^j(s)X_j(\gamma_i(s))-h^j(s)X_j(\gamma(s))\right)\:ds\right|\\
	&\leq \int_0^t\sum_{j=1}^{2} |h_i^j(s)|\left|X_j(\gamma_i(s))-X_j(\gamma(s))\right|\:ds+ \int_0^t\sum_{j=1}^{2} \left|h_i^j(s)-h^j(s)\right| |X_j(\gamma(s))|\:ds.
	\end{aligned}
	\]
	Notice that, by the choice of $h_i^j$, the term
	\[
	\alpha_i(t):=\int_0^t\sum_{j=1}^{2} \left|h_i^j(s)-h^j(s)\right| |X_j(\gamma(s))|\:ds
	\]
	is infinitesimal as $i\to \infty$, and that, by the smoothness of $X_1,X_2$ we can find $C>0$ depending on $T$ and $K$ such that
	\[
	|\gamma_i(t)-\gamma(t)|\leq \alpha_i(t)+CM\int_0^t\left|\gamma_i(s)-\gamma(s)\right|\:ds.
	\]
	We are then in a position to apply Gr\"onwall Lemma to get 
	\[
	|\gamma_i(t)-\gamma(t)|\leq \alpha_i(t)e^{CMt},
	\]
	and letting $i\to \infty$, we conclude the proof of \eqref{eq:num:num10} and in turn of the proposition.
\end{proof}

\section{Construction of a Lipschitz fragment unrectifiable with respect to full curves}

We start this section with a rigidity result for Lipschitz curves that satisfy the hypothesis of \cref{l:curves:Snu}.
\begin{proposizione}\label{propconoXcurv}
Let $\sigma\in (0,1)$ and suppose that $\gamma:[0,T]\to \mathbb F_{2,3}$ is a Lipschitz curve such that $\langle D\gamma(t),X_2\rangle>(1-\sigma^2)\lvert D\gamma(t)\rvert$. Then, for any $t\in [0,T]$ we have
$$
\gamma(s)\in \gamma(t)\cdot\big\{z\in \mathbb F_{2,3}: x_2\geq 0\text{ and }x_2^3x_4-2x_2^2x_3^2-6x_2x_3x_5-6x_5^2\geq 0\big\},\qquad \text{whenever }s\in [t,T],
$$
where we stress that $(x_1,x_2,x_3,x_4,x_5)$ are the exponential coordinates \textbf{of the second type} on $\mathbb F_{2,3}$ associated to the basis $\{X_1,\dots,X_5\}$, see \cite[Section 5]{bellettini2019sets}.
\end{proposizione}
 
\begin{proof}
Without loss of generality, we can assume that $t=0$, otherwise we let $\gamma(\cdot):=\gamma(\cdot+t)\rvert_{[0,T-t]}$. Moreover, up to left translation we can also assume that $\gamma(0)=0$. \cref{l:curves:Snu} shows that for any $s\in [0,T]$ we have that $\gamma(s)\in \mathrm{cl}(X(X_2,\sigma))$ and thus, thanks to \cite[Proposition 5.7]{bellettini2019sets}, we infer that
$$\gamma(s)\in \big\{x\in \mathbb F_{2,3}: x_2\geq 0\text{ and }x_2^3x_4-2x_2^2x_3^2-6x_2x_3x_5-6x_5^2\geq 0\big\}.$$
This concludes the proof of the proposition.
\end{proof}

The following proposition contains the main construction of this note.
\begin{proposizione}\label{curvunr}
There exists a compact set $K\subseteq [0,1]$ of positive $\mathcal{H}^1$-measure and a biLipschitz curve $\gamma:K\to\mathbb{F}_{2,3}$ such that
$$
\gamma(t)=(0,t,0,\gamma_4(t),0),
$$
and where $\gamma_4$ is a striclty decreasing function.
\end{proposizione}

\begin{proof}
Throughout the proof, we will work in exponential coordinates of the second type, see \cite[Section 5]{bellettini2019sets}.
Since the result of the Proposition is invariant up to biLipschitz equivalent distances, we work with the distance introduced in \cref{smoothnorm}.

Before constructing the curve, we need to construct the set $K$. In order to do this, for any $k\in\N$ we let $E_k=\{\mathfrak{J}_{j}^k:j\in\{1,\ldots,2^k\}\}$ be a family of $2^k$ intervals with the following properties
\begin{itemize}
    \item[(i)] $E_1=\{[0,1]\}$,
    \item[(ii)] for any $k\geq2$ and any $1\leq j\leq 2^{k-1}$, defined $c_j^{k-1}$ to be the barycenter of the interval $\mathfrak{J}_{j}^{k-1}$, we have
    \begin{equation}
        \begin{split}
            \mathfrak{J}_{2j-1}^k=&(-\infty,c_j^{k-1})\cap\mathfrak{J}_{j}^{k-1}\setminus (c_j^{k-1}-1/8^{k-1},c_j^{k-1}+1/8^{k-1}),\\
            \mathfrak{J}_{2j}^k=&(c_j^{k-1},\infty)\cap\mathfrak{J}_{j}^{k-1}\setminus (c_j^{k-1}-1/8^{k-1},c_j^{k-1}+1/8^{k-1}).
            \nonumber
        \end{split}
    \end{equation}
\end{itemize}
Note that the requirements (i) and (ii) imply that if $j_2> j_1$ and $t\in \mathfrak{J}_{j_1}^k$, $s\in \mathfrak{J}_{j_2}^k$ then $s> t$.
For any $k$ let
$$
C(k):=\bigcup_{j=1}^{2^k}\mathfrak{J}_j^k.
$$
It is immediate to see that the sets $C(k)$ are nested compact sets such that $\mathcal{H}^1(C(k))\geq 2/3$ and thus the set $K:=\bigcap_{k\in\N}C(k)$
is a compact set for which $\mathcal{H}^1(K)\geq2/3$ thanks to the continuity of the measure from below. 

First of all, let us note that the image of the curve $t\mapsto(0,t,0,0,0)=:\eta(t)$ is isometric to the real line, when $\mathbb{F}_{2,3}$ is endowed with the norm introduced above. Secondly, let $w$ be such that $w_1=0$ and note that
$$w\cdot\eta(t)=(0,w_2+t,w_3,w_4,w_5).$$

We are ready to construct a sequence of Lipschitz functions $\gamma^k:C(k)\to\mathbb F_{2,3}$ such that for any $k\geq 2$ we have
\begin{itemize}
\item[(i)] $\gamma^{k}$ is such that for any $t\in C(k)$ we have
\begin{equation}
    0\leq\gamma^{k-1}_4(t)-\gamma^{k}_4(t)\leq \varepsilon_3^{-3}8^{-6(k-1)},
    \label{M}
\end{equation}
    \item[(ii)] if $t,s\in C(K)$ and $t\neq s$, we have
    %with $j_2\geq j_1$, %then $\gamma_4^{k}(s)\leq\gamma_4^k(t)$ and
    \begin{equation}
      \frac{\varepsilon_3\lvert\gamma_4^k(t)-\gamma_4^k(s)\rvert^{1/3}}{|s-t|}\leq \frac{8^{-1}-8^{-k}}{1-8^{-1}}=:\mathfrak{c}(k).
        \label{P}
    \end{equation}
 %   and moreover if $j_2>j_1$ then $s>t$ and $\gamma_4^k(s)<\gamma_4^k(t)$.
    \item[(iii)]  $j\in \{1,\ldots, 2^{k}\}$ there exists an $\omega(k,j)\in \{-\varepsilon_3^{-3}\sum_{\ell=1}^{k-1} \tau_\ell 8^{-6\ell}:(\tau_1,\ldots,\tau_{k-1})\in \{0,1\}^{k-1} \}=:\mathfrak{T}(k)$ in such a way that $\gamma^{k}\rvert_{\mathfrak{J}_{j}^k}(t)=(0,t,0,\omega(k,j),0)$.
    \item[(iv)]  if $s>t$ belong to different intervals of $E_k$, then 
    $\gamma^k_4(s)-\gamma^k_4(t)\leq -\varepsilon^{-3}_38^{-6(k-1)}$.
\end{itemize}

We construct the $\gamma^k$'s inductively: for $k=1$, we let $\gamma^1:=\eta\rvert_{[0,1]}$, and for a general $k$ we let
\begin{equation}
    \gamma^{k}(t):=\begin{cases}
    \gamma^{k-1}(t) & \text{ if } t\in\mathfrak{J}_{j}^k\textrm{ and $j$ is odd},\\
    \left(\delta_{8^{-2(k-1)}}(\varepsilon_3^{-3}X_4)\right)^{-1}\cdot\gamma^{k-1}(t) & \text{ if } t\in\mathfrak{J}_{j}^k\textrm{ and $j$ is even},
    \end{cases}
\end{equation}
where $\varepsilon_3$ is the constant in the definition of the norm $\|{\cdot}\|$ relative to the third layer. 

Let us prove by induction that the properties (i),(ii), (iii) and (iv) are satisfied by the $\gamma^{k}$'s for any $k\geq 2$. The curve $\gamma^2$ is easily seen to be defined on $[0,3/8]\cup [5/8,1]$ and
\begin{equation}
    \gamma^{2}(t):=\begin{cases}
    (0,t,0,0,0) & \text{ if } t\in[0,3/8],\\
    (0,t,0, -8^{-6}\varepsilon_3^{-3},0) & \text{ if } t\in[5/8,1].
    \end{cases}
\end{equation}
The fact that $\gamma^2$ satisfies items (i), (ii) and (iii), (iv) above is immediate.

Let us now assume that for any $j\in\{1,\ldots,k\}$ we have constructed the curve $\gamma^k$ with the required properties (i), (ii), (iii) and (iv).

Let us note that if $t\in\mathfrak{J}^{k+1}_{j}$ and $j$ is odd, then $\gamma^{k}(t)=\gamma^{k+1}(t)$. Otherwise, if $t\in\mathfrak{J}_{j}^{k+1}$ and $j$ is even, we have by definition of the families $E_k$ that $t\in\mathfrak{J}_{j/2}^k$ and
\begin{equation}
\begin{split}
       \gamma^{k+1}(t)=\left(\delta_{8^{-2k}}(\varepsilon_3^{-3}X_4)\right)^{-1}\cdot\gamma^{k}(t)=&(0,0,0, -8^{-6k}\varepsilon_3^{-3},0)\cdot(0,t,0, \omega(k,j/2),0)
       =(0,t,0, -8^{-6k}\varepsilon_3^{-3}+\omega(k,j/2),0).
       \nonumber
\end{split}
   \end{equation}
The above computation proves simultaneously \eqref{M} and item (iii) since $-8^{-6k}\varepsilon_3^{-3}+\omega(k,j/2)\in \mathfrak{T}(k+1)$.

   Let us prove that item (ii) holds.  Let $s,t\in C(k+1)$ be such that $t,s \in C(k+1)$. If $s,t$ belong to the same interval of $E_{k+1}$, then there is nothing to prove. If this is not the case, since $s$ and $t$ do not belong to the same interval in $E_{k+1}$, we infer that $\lvert s-t\rvert>2\cdot 8^{-k}$. This, together with (ii) applied to $\gamma^{k}$ and the fact \eqref{M} holds for $\gamma^{k+1}$, imply that
   \begin{equation}
   \begin{split}
          \frac{\varepsilon_3\lvert\gamma_4^{k+1}(t)-\gamma_4^{k+1}(s)\rvert^{1/3}}{|s-t|}&\leq\frac{\varepsilon_3\lvert\gamma_4^{k+1}(t)-\gamma_4^{k}(t)\rvert^{1/3}}{|s-t|}+\frac{\varepsilon_3\lvert\gamma_4^{k}(t)-\gamma_4^{k}(s)\rvert^{1/3}}{|s-t|}+\frac{\varepsilon_3\lvert\gamma_4^{k}(s)-\gamma_4^{k+1}(s)\rvert^{1/3}}{|s-t|}\\
          &\leq 2\cdot \frac{8^{-2k}}{2\cdot 8^{-k}}+\mathfrak{c}(k)=\mathfrak{c}(k+1),
   \end{split}
       \label{Pk+1}
   \end{equation}
and this proves \eqref{P} for the curve $\gamma^{k+1}$. Let us now prove item (iv) and complete the induction. Let $s>t$ belong to different intervals of $E_{k+1}$. If $s>t$ belong to the same interval of $E_k$, we have 
$$
\gamma^{k+1}_4(s)-\gamma^{k+1}_4(t)=\gamma^k_4(s)-\gamma^k_4(t)-\varepsilon_3^{-3}8^{-6k}=-\varepsilon_3^{-3}8^{-6k},
$$
and thus (iv) is proved in this case. If $s>t$ belong to different intervals of $E_k$, we have
   \begin{equation}
   \begin{split}
       \gamma^{k+1}_4(s)-\gamma^{k+1}_4(t)=&(\gamma^{k+1}_4(s)-\gamma^k_4(s))+\gamma^k_4(s)-\gamma^k_4(t)+(\gamma^k_4(t)-\gamma^{k+1}_4(t))
       \leq0-\varepsilon_3^{-3}8^{-6(k-1)}+\varepsilon_3^{-3}8^{-6k}\leq -\varepsilon_3^{-3}8^{-6k},
       \nonumber
   \end{split}
   \end{equation}
   where the second last inequality above comes from \eqref{M} and the fact that the inductive hypothesis implies that $\gamma^k$ satisfies the hypothesis (i), (ii), (iii) and (iv). 
   
   Let us now notice that for every $k\geq 0$, we have that $\gamma^{k+1}$ is an isometric immersion of $(C(k+1),\lvert\cdot\rvert_{\mathrm{eu}})$ into $(\mathbb{F}_{2,3},\|\cdot\|)$. Indeed, since  $\mathfrak{c}(k)\leq 1/7$ for any $k\geq 2$, we infer that when $s,t\in C(k+1)$, from item (iv) above we have
\begin{equation}
\begin{split}
        \lVert\gamma^{k+1}(t)^{-1}\cdot\gamma^{k+1}(s)\rVert=&\lVert(0,s-t,0,\gamma_4^{k+1}(s)-\gamma_4^{k+1}(t),0)\rVert=\max\{\lvert s-t\rvert,\varepsilon_3\lvert\gamma_4^{k+1}(s)-\gamma_4^{k+1}(t)\rvert^{1/3}\}=\lvert s-t\rvert.\nonumber
\end{split}
\end{equation}
On the other hand the identity $ \lVert\gamma^{k+1}(t)^{-1}\cdot\gamma^{k+1}(s)\rVert=\lvert s-t\rvert$ is trivially satisfied when $s,t$ belong to the same interval of $E_{k+1}$ thanks to the fact that (iii) holds for $\gamma^{k+1}$. Then the sought claim is proved. %This not only proves that $\gamma^{k+1}$ is $1$-Lipschitz but it proves that it is an isometric immersion of $(C(k+1),\lvert\cdot\rvert_{eu})$ into $\mathbb{F}_{2,3}$.

Restricting the functions $\gamma^{k}$ to $K$, by Ascoli--Arzelà we infer that the curves $\gamma^k\rvert_{K}$ converge uniformly, up to subsequences, to an isometric embedding of $(K,|\cdot|_{\mathrm{eu}})$ in $(\mathbb F_{2,3},\|\cdot\|)$ that we denote by $\gamma$. Also, note that thanks to (iii), we have $\gamma(t)=(0,t,0,\gamma_4(t),0)$.

Let us prove that the function $\gamma_4$ is strictly decreasing. Let $s>t$ be two elements of $K$ and note that by the very definition of $K$ there must exist a $k\in\N$ such that $s$ and $t$ lie in two different elements of $E_k$. Thanks to properties (i), (ii), and (iv) for the curve $\gamma^k$, we infer that
\begin{equation}
    \begin{split}
\varepsilon_3^{-3}8^{-6(k-1)}&\leq \gamma_4^k(t)-\gamma_4^k(s)\leq \lvert \gamma_4(t)-\gamma_4^k(t)\rvert+\gamma_4(t)-\gamma_4(s)+ \lvert \gamma_4(s)-\gamma_4^k(s)\rvert\\
&\leq \sum_{j=k}^{\infty} \lvert \gamma_4^{j+1}(t)-\gamma_4^j(t)\rvert+\gamma_4(t)-\gamma_4(s)+\sum_{j=k}^{\infty} \lvert \gamma_4^{j+1}(s)-\gamma_4^j(s)\rvert\\
&\leq  2\varepsilon_3^{-3}\sum_{j=k}^{\infty} 8^{-6j}+\gamma_4(t)-\gamma_4(s)=\frac{2\varepsilon_3^{-3}8^{-6k}}{1-8^{-6}}+\gamma_4(t)-\gamma_4(s).
\label{numberoevviva}
    \end{split}
\end{equation}
Rearranging terms of the above expression we finally infer that
$$5\varepsilon_3^{-3}8^{-6k}\leq \gamma_4(t)-\gamma_4(s),$$
proving the fact that $\gamma_4$ is strictly decreasing and in turn the proposition.
\end{proof}

From the previous construction we immediately deduce the following 

\begin{proposizione}\label{unrunr}
For any $\sigma>0$ and any full $C(X_2,\sigma)$-curve $\eta$ we have that
$$
\mathrm{Card}(\mathrm{Im}(\gamma)\cap\mathrm{Im}(\eta) )\leq 1,
$$
where $\gamma$ is the curve constructed in \cref{curvunr}.
\end{proposizione}

\begin{proof}
Since $\mathrm{Im}(\gamma)$ and $\mathrm{Im}(\eta)$ are compact sets, the preimage by $\gamma$ in $K$ of their intersection is a compact set $K(\eta)$. We let $t_0:=\min K(\eta)$ and note that thanks to \cref{propconoXcurv}, we have that
\begin{equation}
    \begin{split}
     &\qquad\qquad\qquad\mathrm{Im}(\eta)\cap \mathrm{Im}(\gamma)\overset{(\dagger)}{\subseteq} \gamma(t_0)\cdot\Big\{x\in \mathbb F_{2,3}: x_2\geq 0\text{ and }x_2^3x_4-2x_2^2x_3^2-6x_2x_3x_5-6x_5^2\geq 0\Big\}\\
     =&\Big\{z\in \mathbb F_{2,3}: z_2\geq t_0\text{ and }(z_2-t_0)^3(z_4-\gamma_4(t_0))-2(z_2-t_0)^2z_3^2-6(z_2-t_0)z_3(z_5+t_0z_3)-6(z_5+t_0z_3)^2\underset{(*)}{\geq} 0\Big\},
     \nonumber
    \end{split}
\end{equation}
where the inclusion $(\dagger)$ can be proved by noting that if $\gamma(s_0)=\eta(t_0)$ and $\gamma(s_1)=\eta(t_1)$ and $s_0\leq s_1$ then $t_0\leq t_1$ thanks to the fact that $\eta$ and $\gamma$ are $C(e,\sigma)$-curves. 
In particular we infer for any $t\in K(\eta)\setminus \{ t_0\}$, the coordinates of $\gamma(t)$ must satisfy the inequality ($*$) that in this case boils down to
\begin{equation}
    (t-t_0)^3(\gamma_4(t)-\gamma_4(t_0))\geq 0.
    \nonumber
\end{equation}
Since by construction $t_0$ was the minimum of $K(\eta)$, the above inequality allows us to infer that $\gamma_0(t)\geq\gamma_4(t_0)$, that is in contradiction with the fact that $\gamma_4$ is strictly decreasing. This means that $\gamma(t_0)$ can be the only intersection of the two curves concluding the proof of the proposition.
\end{proof}

Let us now recall the definition of intrinsic Lipschitz graph.
\begin{definizione}[Intrinsic Lipschitz graph]\label{def:IntrinsicLipschitzGraph}
Given $e\in V_1$, a Borel subset $B$ of $\mathfrak{N}(e)$, and a map $\varphi:B\to \mathbb{V}(e)$, we say that $\varphi$ is {\em intrinsic Lipschitz} if $\mathrm{graph}(\varphi):=\{\Phi(t):t\in B\}$
is a $K(e,\sigma)$-set for some $0\leq\sigma<1$, where $\Phi(t):=te\cdot \varphi(te)$ is the graph map of $\varphi$. Note in particular that $\Phi$ satisfies the inequality
$$\dist(\Phi(s),\Phi(t)\mathfrak{N}(X_2))\leq \sigma \lVert\Phi(t)^{-1}\Phi(s)\rVert,$$
for any $s,t\in B$.
Furthermore, if a Borel set $E\subseteq \mathbb{F}_{2,3}$ is the graph of an intrinsic Lipschitz, we will call it \emph{intrinsic Lipschitz graph}.
\end{definizione}

Note that the curve $P_{\mathbb{V}(e)}(\gamma)=\gamma_4$ constructed in Proposition \ref{curvunr} is an intrinsic Lipschitz map, indeed for any $s,t\in K$ we have
\begin{equation}
    \begin{split}
        \dist(\gamma(s),\gamma(t)\mathfrak{N}(X_2))=&\inf_{\lambda\in\R}\lVert(0,t-s-\lambda,0,\gamma_4(t)-\gamma_4(s),0)\rVert\\
        =&\varepsilon_3\lvert \gamma_4(t)-\gamma_4(s)\rvert^{1/3}\overset{\eqref{P}}{\leq}\lvert s-t\rvert/7\leq\lVert\gamma(t)^{-1}\gamma(s)\rVert/7.
    \end{split}
\end{equation}
Furthermore, if we let $\Phi:I\subseteq \mathfrak{N}(X_2)\to \mathbb{F}_{2,3}$, where $I$ is a compact interval, to be the graph map of an intrinsic Lipschitz function where the cones have opening $\sigma$, it is immediate to see that $\Phi$ is a Lipschitz map from the compact interval $I$ into $\mathbb{F}_{2,3}$, and thus for $\mathcal{H}^1$-almost any $t\in I$ we have
\begin{equation}
    \begin{split}
        \dist(D\Phi(t),\mathfrak{N}(X_2))=&\lim_{s\to t}    \dist(\delta_{1/\lvert s-t\rvert}(\Phi(t)^{-1}\Phi(s)),\mathfrak{N}(X_2))=\lim_{s\to t}\frac{\dist(\Phi(t)^{-1}\Phi(s),\mathfrak{N}(X_2))}{\lvert s-t\rvert}\\
        \leq&\liminf_{s\to t}\frac{\sigma\|\Phi(t)^{-1}\Phi(s)\|}{\lvert s-t\rvert}=\sigma \| D\Phi(t)\|.
    \end{split}
\end{equation}
Since $0<\sigma<1$, then $\| D\Phi(t)\|>0$ and this, thanks to few omitted elementary algebraic computations, shows that $\Phi$ is a full $C(X_2,\sqrt{1-\sqrt{1-\sigma^2}})$-curve. From the previous reasoning and from \cref{unrunr} we deduce that that $\mathcal{H}^1(\mathrm{Im}(\gamma)\cap \mathrm{Im}(\Phi))=0$. Hence, by writing every open set in $\mathbb R$ as a union of countably many compact intervals, the above argument together with \cref{unrunr}, yields the main result of this note.

\begin{teorema}\label{prop.intersez.curves}
There exists an intrinsic Lipschitz $X_2$-graph $E$ in $\mathbb{F}_{2,3}$ such that for any intrinsic Lipschitz map $\varphi:\Omega\to \mathbb{V}(X_2)$, where $\Omega$ is an open subset of $\mathfrak{N}(X_2)$, we have
$$\mathcal{H}^1(E\cap \mathrm{graph}(\varphi))=0.$$
\end{teorema}

Let us now show that the example provided by \cref{prop.intersez.curves} exists in every horizontal direction. The forthcoming theorem, together with the argument above \cref{prop.intersez.curves}, will enable us to prove the analogous statement of \cref{prop.intersez.curves} obtained substituting $X_2$ with an arbitrary unit vector $e\in V_1$. This in turn will conclude the proof of \cref{thm:EveryDirectionINTRO}.
\begin{teorema}\label{thm:EveryDirection}
For any unit vector $e\in V_1$ there exists a biLipschitz curve $\gamma_e:K\to \mathbb{F}_{2,3}$, where $K$ is the compact set constructed in Proposition \ref{curvunr}, such that
\begin{itemize}
    \item[(i)]$D\gamma_e=e$ for $\mathcal{H}^1$-almost every $x\in K$,
    \item[(ii)] for any $\sigma\in(0,1)$ and any full $C(e,\sigma)$-curve $\eta$ we have $\mathcal{H}^1(\mathrm{Im}(\gamma_e)\cap\mathrm{Im}(\eta))=0$.
\end{itemize}
\end{teorema}

\begin{proof}
    Let $e^\perp$ be the orthogonal unit vector of $e$, and let $L$ to be the orthogonal transformation that sends the basis $\{X_1,X_2\}$ to $\{e^\perp,e\}$. Let 
$\Psi$ be the automorphism associated to $L$ yielded by Proposition \ref{pautomorf}.
We claim that the curve $\gamma_e:=\Psi(\gamma)$ satisfies items (i) and (ii). Item (i) follows from the following computation
\begin{equation}
\begin{split}
       D\gamma_e(t)=&\lim_{h\to 0,\,t+h\in K}\delta_{1/h}(\gamma_e(t)^{-1}\gamma_e(t+h))=\lim_{h\to 0,\,t+h\in K}\Psi(\delta_{1/h}(\gamma(t)^{-1}\gamma(t+h)))=\Psi(D\gamma(t))=\Psi(X_2)=e.
       \nonumber
\end{split}
\end{equation}
Now let $\eta$ be a full $C(e,\sigma)$-curve and note that
\begin{equation}
    \mathcal{H}^1(\mathrm{Im}(\gamma_e)\cap\mathrm{Im}(\eta))=\mathcal{H}^1(\Psi(\mathrm{Im}(\gamma)\cap\mathrm{Im}(\Psi^{-1}\circ\eta)))\leq\mathrm{Lip}(\Psi)\mathcal{H}^1(\mathrm{Im}(\gamma)\cap\mathrm{Im}(\Psi^{-1}\circ\eta)),
    \nonumber
\end{equation}
where the last inequality comes from the fact that every group homomorphism is a Lipschitz map.
If we prove that the curve $\Psi^{-1}\circ \eta$ is a $C(X_2,\sigma')$-curve, for some $\sigma'$, we have proved the proposition. First, the fact that $\Psi$, being a homogeneous automorphism, is biLipschitz implies that $\Psi^{-1}\circ \eta$ is a Lipschitz full curve. Moreover, let us compute
\begin{equation}
\begin{split}
       \langle D\Psi^{-1}\eta(t),X_2\rangle&=\Big\langle\lim_{h\to 0,\,t+h\in K}\delta_{1/h}(\Psi^{-1}(\eta(t)^{-1}\eta(t+h))), X_2\Big\rangle\\
       &=\Big\langle\lim_{h\to 0,\,t+h\in K}\Psi^{-1}(\delta_{1/h}(\eta(t)^{-1}\eta(t+h))), X_2\Big\rangle=\big\langle\Psi^{-1}(D\eta(t)), X_2\big\rangle=\langle D\eta(t),e\rangle> (1-\sigma^2)\lvert D\eta(t)\rvert\\
       &\geq (1-\sigma'^2)|D\Psi^{-1}\eta(t)|,
       \nonumber
\end{split}
\end{equation}
for some $\sigma'$, where in the last inequality we are using that $\Psi$ is biLipschitz. Hence we have proved 
the claim,
% and the fundamental theorem of calculus concludes the proof of the fact that $D\Psi^{-1}\eta$ is a full $C(X_2,\sigma)$-curve. 
and we can finally conclude
$$
\mathcal{H}^1(\mathrm{Im}(\gamma_e)\cap\mathrm{Im}(\eta))\leq \mathrm{Lip}(\Psi^{-1})\mathcal{H}^1(\mathrm{Im}(\gamma)\cap\mathrm{Im}(\Psi^{-1}\circ\eta))=0,
$$
where the last identity comes from Proposition \ref{unrunr}.
\end{proof}

\begin{osservazione}\label{rem:Engel}
Let us recall that the Engel group $\mathbb E$ is the Carnot group whose Lie algebra is 
$$
\mathfrak e:=\mathrm{span}\{X_1,X_2,X_3,X_4\},
$$
where the only nonvanishing bracket relations are $[X_1,X_2]=X_3$, $[X_1,X_3]=X_4$.
Taking into account the results in \cite[Example 3.31, Remark 3.32]{BLD1} and by using exponential coordinates of the second type, we have that in the Engel group $\mathbb E$, \cref{propconoXcurv} holds with the cone 
$$
\{z\in \mathbb E:x_2\geq 0,x_4\geq 0,2x_2x_4-x_3^2\geq 0\}.
$$
Hence, if we take $\gamma_4:K\to\mathbb R$ constructed as in \cref{curvunr}, changing $\varepsilon_3$ accordingly in order to have a norm like the one in \cref{smoothnorm} in the Engel group, we have that the curve 
$$
\widetilde\gamma(t):=(0,t,0,\gamma_4(t)),
$$
still satisfies \cref{unrunr}. All in all, one gets the same result as in \cref{prop.intersez.curves} in the Engel group $\mathbb E$ and in the direction $X_2$.

Notice that in the Engel group $\mathbb E$ we cannot adapt the strategy in \cref{thm:EveryDirection} in order to provide an example where the horizontal derivative is an arbitrary element in $V_1$. This is due to the fact that if $\Psi$ is a group automorphism that sends $\{X_1,X_2\}$ to another basis $\{Y_1,Y_2\}$ of $V_1$, we have that $Y_2=bX_2$ for some $b\neq 0$. This is indeed linked to the fact the unique horizontal abnormal direction in Engel is $X_2$, see \cite[Proposition 6.4]{DLDMV19}.

%{\color{red} We claim that any $X$-iLip graph, where $X$ is a horizontal non-abnormal direction in Engel, is extendible or at least it can be covered by countably many full iLip graphs.}

\end{osservazione}

\printbibliography

@phdthesis{tesimonti,
AUTHOR={Monti, Roberto},
TITLE={Distances, boundaries and surface measures in Carnot-Carath\'eodory spaces},
SCHOOL={Universit\'a degli studi di Trento},
YEAR={2001},
}

@misc{Foliated20,
    title={Foliated corona decompositions},
    author={Assaf Naor and Robert Young},
    year={2020},
    howpublished={Preprint on arXiv, arXiv:2004.12522.},
}

@misc{bellettini2019sets,
    title={Sets with constant normal in Carnot groups: properties and examples},
    author={Costante Bellettini and Enrico Le Donne},
    year={2019},
    howpublished={Preprint on arXiv, arXiv:1910.12117. Accepted for publication in "Commentarii Mathematici Helvetici"},
}

@article {BLD1,
    AUTHOR = {Bellettini, Costante and Le Donne, Enrico},
     TITLE = {Regularity of sets with constant horizontal normal in the
              {E}ngel group},
   JOURNAL = {Comm. Anal. Geom.},
  FJOURNAL = {Communications in Analysis and Geometry},
    VOLUME = {21},
      YEAR = {2013},
    NUMBER = {3},
     PAGES = {469--507},
      ISSN = {1019-8385},
   MRCLASS = {53C17 (22Exx 49Q05 53C40)},
  MRNUMBER = {3078944},
MRREVIEWER = {Davide Vittone},
       DOI = {10.4310/CAG.2013.v21.n3.a1},
       URL = {https://doi.org/10.4310/CAG.2013.v21.n3.a1},
}

@misc{antonelli2020rectifiable,
      title={On rectifiable measures in Carnot groups: structure theory}, 
      author={Gioacchino Antonelli and Andrea Merlo},
      year={2020},
      eprint={2009.13941},
      archivePrefix={arXiv},
      primaryClass={math.MG}
}

@article {kircharea,
    AUTHOR = {Kirchheim, Bernd},
     TITLE = {Rectifiable metric spaces: local structure and regularity of
              the {H}ausdorff measure},
   JOURNAL = {Proc. Amer. Math. Soc.},
  FJOURNAL = {Proceedings of the American Mathematical Society},
    VOLUME = {121},
      YEAR = {1994},
    NUMBER = {1},
     PAGES = {113--123},
      ISSN = {0002-9939},
   MRCLASS = {28A78},
  MRNUMBER = {1189747},
MRREVIEWER = {G. Freilich},
       DOI = {10.2307/2160371},
       URL = {https://doi.org/10.2307/2160371},
}

@ARTICLE{AntonelliMerlo2021,
       author = {{Antonelli}, Gioacchino and {Merlo}, Andrea},
        title = "{On rectifiable measures in Carnot groups: representation}",
      journal = {arXiv e-prints},
         year = 2021,
        month = apr,
          eid = {arXiv:2104.00335},
        pages = {arXiv:2104.00335},
archivePrefix = {arXiv},
       eprint = {2104.00335},
 primaryClass = {math.MG},
       adsurl = {https://ui.adsabs.harvard.edu/abs/2021arXiv210400335A}
}

@misc{AM20,
author = {Antonelli, Gioacchino and Merlo, Andrea},
title={Intrinsically Lipschitz functions with normal targets in Carnot groups},
year={2020},
howpublished = {Preprint on arXiv, 	arXiv:2006.02782. Accepted for publication in "Annales Academiæ Scientiarum Fennicæ".}
}

@misc{DLDMV19,
author = {Don, Sebastiano and Le Donne, Enrico and Moisala, Terhi and Vittone, Davide},
title={A rectifiability result for finite-perimeter sets in Carnot groups},
year={2019},
howpublished = {Preprint on arXiv, 	arXiv:1912.00493. Accepted for publication in "Indiana University Mathematics Journal".},
}

@book{Federer1996GeometricTheory,
    AUTHOR = {Federer, Herbert},
     TITLE = {Geometric measure theory},
    SERIES = {Die Grundlehren der mathematischen Wissenschaften, Band 153},
 PUBLISHER = {Springer-Verlag New York Inc., New York},
      YEAR = {1969},
     PAGES = {xiv+676},
   MRCLASS = {28.80 (26.00)},
  MRNUMBER = {0257325},
MRREVIEWER = {J. E. Brothers},
}

@misc{Vittone20, 
    AUTHOR = {Vittone, Davide},
     TITLE = {Lipschitz graphs and currents in Heisenberg groups},
howpublished = {Preprint on arXiv, 	arXiv:2007.14286}
}

@article{Serapioni2001RectifiabilityGroup,
  AUTHOR = {Franchi, Bruno and Serapioni, Raul and Serra Cassano,
              Francesco},
     TITLE = {Rectifiability and perimeter in the {H}eisenberg group},
   JOURNAL = {Math. Ann.},
  FJOURNAL = {Mathematische Annalen},
    VOLUME = {321},
      YEAR = {2001},
    NUMBER = {3},
     PAGES = {479--531},
      ISSN = {0025-5831},
   MRCLASS = {49Q15 (22E25 46E35)},
  MRNUMBER = {1871966},
MRREVIEWER = {Piotr Haj\l asz},
       DOI = {10.1007/s002080100228},
       URL = {https://doi.org/10.1007/s002080100228},
}

@article {MagnaniUnrect,
    AUTHOR = {Magnani, Valentino},
     TITLE = {Unrectifiability and rigidity in stratified groups},
   JOURNAL = {Arch. Math. (Basel)},
  FJOURNAL = {Archiv der Mathematik},
    VOLUME = {83},
      YEAR = {2004},
    NUMBER = {6},
     PAGES = {568--576},
      ISSN = {0003-889X},
   MRCLASS = {53C17 (28A75 53C24)},
  MRNUMBER = {2105335},
MRREVIEWER = {Gerald B. Folland},
       DOI = {10.1007/s00013-004-1057-4},
       URL = {https://doi.org/10.1007/s00013-004-1057-4},
}

@article {FranchiSerapioni16,
    AUTHOR = {Franchi, Bruno and Serapioni, Raul Paolo},
     TITLE = {Intrinsic {L}ipschitz graphs within {C}arnot groups},
   JOURNAL = {J. Geom. Anal.},
  FJOURNAL = {Journal of Geometric Analysis},
    VOLUME = {26},
      YEAR = {2016},
    NUMBER = {3},
     PAGES = {1946--1994},
      ISSN = {1050-6926},
   MRCLASS = {49Q15 (22E25 53C17 58C20 58C35)},
  MRNUMBER = {3511465},
MRREVIEWER = {Jingzhi Tie},
       DOI = {10.1007/s12220-015-9615-5},
       URL = {https://doi.org/10.1007/s12220-015-9615-5},
}

@article {step2,
    AUTHOR = {Franchi, Bruno and Serapioni, Raul and Serra Cassano,
              Francesco},
     TITLE = {On the structure of finite perimeter sets in step 2 {C}arnot
              groups},
   JOURNAL = {J. Geom. Anal.},
  FJOURNAL = {The Journal of Geometric Analysis},
    VOLUME = {13},
      YEAR = {2003},
    NUMBER = {3},
     PAGES = {421--466},
      ISSN = {1050-6926},
   MRCLASS = {49Q15 (53C17)},
  MRNUMBER = {1984849},
MRREVIEWER = {J. E. Brothers},
       DOI = {10.1007/BF02922053},
       URL = {https://doi.org/10.1007/BF02922053},
}

@article {BaloghFassler09,
    AUTHOR = {Balogh, Zolt\'{a}n M. and F\"{a}ssler, Katrin S.},
     TITLE = {Rectifiability and {L}ipschitz extensions into the
              {H}eisenberg group},
   JOURNAL = {Math. Z.},
  FJOURNAL = {Mathematische Zeitschrift},
    VOLUME = {263},
      YEAR = {2009},
    NUMBER = {3},
     PAGES = {673--683},
      ISSN = {0025-5874},
   MRCLASS = {53C17 (22E30 49Q20 53C23)},
  MRNUMBER = {2545863},
MRREVIEWER = {Jeremy T. Tyson},
       DOI = {10.1007/s00209-008-0437-z},
       URL = {https://doi.org/10.1007/s00209-008-0437-z},
}

@article {WengerYoung10,
    AUTHOR = {Wenger, Stefan and Young, Robert},
     TITLE = {Lipschitz extensions into jet space {C}arnot groups},
   JOURNAL = {Math. Res. Lett.},
  FJOURNAL = {Mathematical Research Letters},
    VOLUME = {17},
      YEAR = {2010},
    NUMBER = {6},
     PAGES = {1137--1149},
      ISSN = {1073-2780},
   MRCLASS = {53C17 (58A20)},
  MRNUMBER = {2729637},
MRREVIEWER = {Davide Vittone},
       DOI = {10.4310/MRL.2010.v17.n6.a12},
       URL = {https://doi.org/10.4310/MRL.2010.v17.n6.a12},
}

@misc{KatrinDaniela,
    title={Extensions and corona decompositions of low-dimensional intrinsic Lipschitz graphs in Heisenberg groups.},
    author={Daniela di Donato and Katrin Fässler},
    year={2021},
    howpublished={Preprint on arXiv, arXiv:2012.12609. Accepted for publication in "Annali di Matematica Pura ed Applicata".}
}

@article {Magnani10,
    AUTHOR = {Magnani, Valentino},
     TITLE = {Contact equations, {L}ipschitz extensions and isoperimetric
              inequalities},
   JOURNAL = {Calc. Var. Partial Differential Equations},
  FJOURNAL = {Calculus of Variations and Partial Differential Equations},
    VOLUME = {39},
      YEAR = {2010},
    NUMBER = {1-2},
     PAGES = {233--271},
      ISSN = {0944-2669},
   MRCLASS = {49Q20},
  MRNUMBER = {2659687},
MRREVIEWER = {Ana Hurtado},
       DOI = {10.1007/s00526-010-0309-3},
       URL = {https://doi.org/10.1007/s00526-010-0309-3},
}

@article {FSSC11,
    AUTHOR = {Franchi, Bruno and Serapioni, Raul and Serra Cassano,
              Francesco},
     TITLE = {Differentiability of intrinsic {L}ipschitz functions within
              {H}eisenberg groups},
   JOURNAL = {J. Geom. Anal.},
  FJOURNAL = {Journal of Geometric Analysis},
    VOLUME = {21},
      YEAR = {2011},
    NUMBER = {4},
     PAGES = {1044--1084},
      ISSN = {1050-6926},
   MRCLASS = {22E30 (58C20)},
  MRNUMBER = {2836591},
MRREVIEWER = {Davide Vittone},
       DOI = {10.1007/s12220-010-9178-4},
       URL = {https://doi.org/10.1007/s12220-010-9178-4},
}

@article {Pansu,
    AUTHOR = {Pansu, Pierre},
     TITLE = {M\'{e}triques de {C}arnot-{C}arath\'{e}odory et quasiisom\'{e}tries des
              espaces sym\'{e}triques de rang un},
   JOURNAL = {Ann. of Math. (2)},
  FJOURNAL = {Annals of Mathematics. Second Series},
    VOLUME = {129},
      YEAR = {1989},
    NUMBER = {1},
     PAGES = {1--60},
      ISSN = {0003-486X},
   MRCLASS = {53C20 (22E40)},
  MRNUMBER = {979599},
MRREVIEWER = {Gudlaugur Thorbergsson},
       DOI = {10.2307/1971484},
       URL = {https://doi.org/10.2307/1971484},
}

@article {FSSC06,
    AUTHOR = {Franchi, Bruno and Serapioni, Raul and Serra Cassano,
              Francesco},
     TITLE = {Intrinsic {L}ipschitz graphs in {H}eisenberg groups},
   JOURNAL = {J. Nonlinear Convex Anal.},
  FJOURNAL = {Journal of Nonlinear and Convex Analysis. An International
              Journal},
    VOLUME = {7},
      YEAR = {2006},
    NUMBER = {3},
     PAGES = {423--441},
      ISSN = {1345-4773},
   MRCLASS = {58C20 (22E30)},
  MRNUMBER = {2287539},
MRREVIEWER = {Thierry Coulhon},
}

@article {LD17,
    AUTHOR = {Le Donne, Enrico},
     TITLE = {A primer on {C}arnot groups: homogenous groups,
              {C}arnot-{C}arath\'{e}odory spaces, and regularity of their
              isometries},
   JOURNAL = {Anal. Geom. Metr. Spaces},
  FJOURNAL = {Analysis and Geometry in Metric Spaces},
    VOLUME = {5},
      YEAR = {2017},
    NUMBER = {1},
     PAGES = {116--137},
   MRCLASS = {53C17 (22E25 22F30 43A80)},
  MRNUMBER = {3742567},
MRREVIEWER = {Andrea Pinamonti},
       DOI = {10.1515/agms-2017-0007},
       URL = {https://doi.org/10.1515/agms-2017-0007},
}

@article {AK00,
    AUTHOR = {Ambrosio, Luigi and Kirchheim, Bernd},
     TITLE = {Rectifiable sets in metric and {B}anach spaces},
   JOURNAL = {Math. Ann.},
  FJOURNAL = {Mathematische Annalen},
    VOLUME = {318},
      YEAR = {2000},
    NUMBER = {3},
     PAGES = {527--555},
      ISSN = {0025-5831},
   MRCLASS = {28A75 (46G99 46T99 49Q20)},
  MRNUMBER = {1800768},
MRREVIEWER = {Piotr Haj\l asz},
       DOI = {10.1007/s002080000122},
       URL = {https://doi.org/10.1007/s002080000122},
}

@article {ChousionisFasslerOrponen19,
    AUTHOR = {Chousionis, Vasileios and F\"{a}ssler, Katrin and Orponen, Tuomas},
     TITLE = {Intrinsic {L}ipschitz graphs and vertical {$\beta$}-numbers in
              the {H}eisenberg group},
   JOURNAL = {Amer. J. Math.},
  FJOURNAL = {American Journal of Mathematics},
    VOLUME = {141},
      YEAR = {2019},
    NUMBER = {4},
     PAGES = {1087--1147},
      ISSN = {0002-9327},
   MRCLASS = {49Q15 (28A75)},
  MRNUMBER = {3992573},
       DOI = {10.1353/ajm.2019.0028},
       URL = {https://doi.org/10.1353/ajm.2019.0028},
}

@article {MR852474,
    AUTHOR = {Johnson, William B. and Lindenstrauss, Joram and Schechtman,
              Gideon},
     TITLE = {Extensions of {L}ipschitz maps into {B}anach spaces},
   JOURNAL = {Israel J. Math.},
  FJOURNAL = {Israel Journal of Mathematics},
    VOLUME = {54},
      YEAR = {1986},
    NUMBER = {2},
     PAGES = {129--138},
      ISSN = {0021-2172},
   MRCLASS = {54C20 (46B99 54H99)},
  MRNUMBER = {852474},
MRREVIEWER = {Gilles Pisier},
       DOI = {10.1007/BF02764938},
       URL = {https://doi.org/10.1007/BF02764938},
}

\end{document}